\documentclass[11pt, oneside]{article}  
\usepackage{geometry}                		
\usepackage{amssymb}
\usepackage{amsthm}
\usepackage{graphicx}
\usepackage{amssymb}
\usepackage{mathtools}
\usepackage{wrapfig}
\usepackage{color}
\usepackage{float}
\usepackage{cuted, nccmath}

\newtheorem{assumption}{Assumption}
\newtheorem{theorem}{Theorem}

\newtheorem{lemma}{Lemma}
\newtheorem{claim}{Claim}

\allowdisplaybreaks

\title{Model order reduction for Linear Noise Approximation using time-scale separation (Extended Version)}
\author{Narmada Herath%
\thanks{Narmada Herath is with the Department of Electrical Engineering and Computer Science, Massachusetts Institute of Technology, 77 Mass. Ave, Cambridge MA
        {\tt\small nherath@mit.edu}}%
        \ \ and Domitilla Del Vecchio%
\thanks{Domitilla Del Vecchio is with the Department of Mechanical Engineering, Massachusetts Institute of Technology, 77 Mass. Ave, Cambridge MA   {\tt\small ddv@mit.edu}}%
}
\date{}							

\begin{document}
\maketitle
\begin{abstract}


In this paper, we focus on model reduction of biomolecular systems with multiple time-scales, modeled using the Linear Noise Approximation. Considering systems where the Linear Noise Approximation can be written in singular perturbation form, with $\epsilon$ as the singular perturbation parameter, we obtain a reduced order model that approximates the slow variable dynamics of the original system. In particular, we show that, on a finite time-interval, the first and second moments of the reduced system are within an $O(\epsilon)$-neighborhood of the first and second moments of the slow variable dynamics of the original system. The approach is illustrated on an example of a biomolecular system that exhibits time-scale separation. 
\end{abstract}

\section{Introduction}
\normalsize
Time-scale separation is a ubiquitous feature in biomolecular systems, which enables the separation of the system dynamics into `slow' and `fast'. This property is widely used in biological applications to reduce the complexity in dynamical models. In deterministic systems, where the dynamics are modeled using ordinary differential equations, the process of obtaining a reduced model is well defined by singular perturbation and averaging techniques \cite{khalil, pavliotis}.  However, employing time-scale separation to obtain a reduced order model remains an ongoing area of research for stochastic models of biological systems \cite{kim2014validity}. 

Biological systems are inherently stochastic due to randomness in chemical reactions \cite{andrews2009stochastic, mcquarrie1967stochastic}. Thus, different stochastic models have been developed to capture the randomness in the system dynamics, especially at low population numbers.  The chemical Master equation is a prominent stochastic model which considers the species counts as a set of discrete states and provides a description for the time-evolution of their probability density functions \cite{gardiner1985handbook, kampen}. However, analyzing the chemical Master equation directly proves to be a challenge due to the lack of analytical tools to analyze its behavior. Therefore, several approximations of the Master equation have been developed, which provide good descriptions of the system dynamics under certain assumptions. The chemical Langevin equation (CLE) is one such approximation, where the dynamics of the chemical species are described as a set of stochastic differential equations \cite{gillespie}. The Fokker-Plank equation is another method equivalent to the CLE, which considers the species counts as continuous variables and provides a description of the time evolution of their probability density functions \cite{gardiner1985handbook}. The Linear Noise Approximation (LNA) is another approximation, where the system dynamics are portrayed as stochastic fluctuations about a deterministic trajectory, assuming that the system volume is sufficiently large such that the fluctuations are small relative to the average species counts \cite{kampen, elf2003fast}.

In our previous work, we considered a class of stochastic differential equations in singular perturbation form, which captures the case of multiple scale chemical Langevin equation with linear propensity functions. We obtained a reduced order model for which the error between the moment dynamics were of $O(\epsilon)$, where $\epsilon$ is the singular perturbation parameter \cite{herathACC2015, AUCC2015, herathACC2016}. In this work, we consider systems with nonlinear propensity functions, modeled using the Linear Noise Approximation.

There have been several works that obtain reduced order models for systems modeled using LNA, under different approaches for time-scale separation. One such model is derived by Pahlajani et. al, in \cite{pahlajani2011stochastic}, where the slow and fast variables are identified by categorizing the chemical reactions as slow and fast. In \cite{thomas2012rigorous, thomas2012slow}, Thomas et. al, derive a reduced order model by considering the case where the species are separated using the decay rate of their transients, according to the quasi-steady-state approximation for chemical kinetics. It is also shown that, imposing the time-scale separation conditions arising from slow and fast reactions, on their model, leads to the same reduced model obtained in \cite{pahlajani2011stochastic}.
In these previous works, the error between the original system and the reduced system has been studied numerically and has not been analytically quantified. The work by Sootla and Anderson in \cite{sootla2014CDC} gives a projection-based model order reduction method for systems modeled by the Linear Noise Approximation. This work is extended in \cite{sootla2015structured} by the same authors, where they also provide an error quantification in mean square sense for the reduced order model derived in \cite{thomas2012rigorous} under quasi-steady state assumptions. However, to provide an error bound the authors explicitly use the Lipschitz continuity of the diffusion term, which is not Lipschitz continuous in general. 

In this paper, we consider biomolecular systems modeled using the Linear Noise Approximation where system dynamics are represented by a set of ordinary differential equations that give the deterministic trajectory and a set of stochastic differential equations that describe the stochastic fluctuations about the deterministic trajectory. We consider the case where the system dynamics evolve on well separated time-scales with slow and fast reactions, and the LNA can be written in singular perturbation form with $\epsilon$ as the singular perturbation parameter, as in \cite{pahlajani2011stochastic}. We define a reduced order model and prove that the first and second moments of the reduced system are within an $O(\epsilon)$-neighborhood of the first and second moments of the original system. Our results do not rely on Lipschitz continuity assumptions on the diffusion term of the LNA.

This paper is organized as follows. In Section \ref{model}, we describe the model considered. In Section \ref{pre_results}, we define the reduced system and derive the moment dynamics for the original and reduced systems.  In Section \ref{main_results}, we prove the main convergence results. Section \ref{examples} illustrates our approach with an example and Section \ref{conclusion} includes the concluding remarks.

\section{System Model}
\label{model}
\subsection{Linear Noise Approximation}
Consider a biomolecular system with $n$ species interacting through $m$ reactions in a given volume $\Omega$. The Chemical Master Equation (CME) describes the evolution of the probability distribution for the species counts to be in state ${Y} = (Y_1, \ldots, Y_N)$, by the ordinary differential equation 
\begin{align}
	\frac{\partial P(Y,t)}{\partial t} = \sum_{i = 1}^{m} [{a}_i(Y - v_i, t) P(Y - v_i, t) - {a}_i(Y, t) P(Y, t)],
\end{align}
where ${a}_i(Y,t)$ is the microscopic reaction rate with ${a}_i(Y,t) dt$ being the probability that a reaction $i$ will take place in an infinitesimal time step $dt$ and $v_i$ the change in state produced by reaction $i$ for $i = 1, \ldots, m$ \cite{gillespie2007stochastic}. 

The Linear Noise Approximation (LNA) is an approximation to the CME obtained under the assumption that the system volume $\Omega$ and the number of molecules in the system are large \cite{kampen}. To derive the LNA it is assumed that $Y = \Omega y + \sqrt{\Omega} \xi$, where $y$ is a deterministic quantity and $\xi$ is a stochastic variable accounting for the stochastic fluctuations.  Then by expanding the chemical Master equation in a Taylor series and equating the terms of order $\Omega^{1/2}$ and $\Omega^{0}$, it is shown that $y$ is the macroscopic concentration and $\xi$ is a Gaussian process whose dynamics are given by   \cite{kampen, elf2003fast}  
\begin{align}
\label{LNA_intro1}	\dot{y} &= f(y,t),\\
\label{LNA_intro2}	\dot{\xi} &= A(y,t) \xi + \sigma(y,t) \Gamma,
\end{align}
where $\Gamma$ is an $m$-dimensional white noise process, $f(y,t) = \sum_{i = 1}^m v_i \tilde{a}_i(y,t)$, $A(y,t) = \frac{\partial f(y,t)}{\partial y}$ and $\sigma(y,t) = [v_1 \sqrt{\tilde{a}_1(y,t)}, \ldots, v_m \sqrt{\tilde{a}_m(y,t)}]$. $\tilde{a}_i(y,t)$ is the macroscopic reaction rate which can be approximated by $\tilde{a}_i(y,t) = \frac{1}{\Omega}{a}_i(\Omega y,t)$ at the limit of $\Omega \to \infty$ and  $Y \to \infty$ such that the concentration $y = Y/\Omega$ remains constant \cite{gillespie2009deterministic}.

\subsection{Singularly Perturbed System}

We consider the case where the biomolecular system in (\ref{LNA_intro1}) - (\ref{LNA_intro2}) exhibits time-scale separation, with $m_s$ slow reactions and $m_f$ fast reactions where $m_s + m_f = m$. This allows the use of a small parameter $\epsilon$ to decompose the reaction rate vector as $\tilde{a}(y,t) = [\hat{a}_s(y,t), (1/\epsilon)\hat{a}_f(y,t)]^T$ where $\hat{a}_s(y,t) \in  \mathbb{R}^{m_s} $ represents the reaction rates for the slow reactions and $(1/\epsilon)\hat{a}_f(y,t) \in \mathbb{R}^{m_f}$ represents the reaction rates for the fast reactions. The corresponding $v_i$ vectors representing the change of state by each reaction $i$ could be represented as $v = [v_1, \ldots, v_{m_s}, v_{m_s + 1}, \ldots, v_{m_s + m_f}]$ for $m_s$ slow and $m_f$ fast reactions. 
However, such a decomposition does not guarantee that the individual species in the system will evolve on well-separated time-scales. Therefore, a coordinate transformation may be necessary to identify the slow and fast variables in the system as seen in deterministic systems  \cite{jayanthi2011retroactivity} and chemical Langevin models \cite{contou2011model}. Thus, we make the following claim.

\begin{claim}
Assume there is an invertible matrix $A = [A_x, A_z]^T$ with $A_x \in \mathbb{R}^{n_s \times n}$ and $A_z \in \mathbb{R}^{n_f \times n}$, such that the change of variables $x = A_xy$, $z= A_zy$, allows the deterministic dynamics in (\ref{LNA_intro1}) to be written in the singular perturbation form 
\begin{align}
\label{sys_ori1_cl}		\dot{x} &= f_x(x,z,t),\\
\label{sys_ori2_cl}		\epsilon \dot{z} &= f_z(x,z,t,\epsilon).	
\end{align}
Then, the change of variables $\psi_x = A_x\xi$, $\psi_z= A_z\xi$ takes the dynamics of the stochastic fluctuations given in (\ref{LNA_intro2}), in to the singular perturbation form
\begin{align}
\label{sys_ori3_cl}		\dot{\psi_x} &= A_{1}(x,z,t) \psi_x + A_2(x,z,t) \psi_z + \sigma_x(x,z,t) \Gamma_x,\\
\label{sys_ori4_cl}		\epsilon \dot{\psi_z} &= B_1(x,z,t,\epsilon) \psi_x + B_2(x,z,t,\epsilon) \psi_z + \sigma_z(x,z,t, \epsilon) \Gamma_z,	
\end{align}
where $\Gamma_x$ is an $m_s$-dimensional white noise process, $\Gamma_z = [\Gamma_x, \Gamma_f]^T, $  where $\Gamma_f$ is an $m_f$-dimensional white noise process and  
\small
\begin{align*}
	&A_{1}(x,z,t)  = \frac{\partial  f_x(x,z,t)}{\partial x}, \\
	&A_{2}(x,z,t)  = \frac{\partial  f_x(x,z,t)}{\partial z}, \\
	&B_1(x,z,t,\epsilon)  = \frac{\partial  f_z(x,z,t,\epsilon)}{\partial x}, \\
	&B_2(x,z,t,\epsilon)  = \frac{\partial  f_z(x,z,t,\epsilon)}{\partial z}, \\
	 &\sigma_x(x,z,t) =   A_x\left [v_1\sqrt{\hat{a}{_s}_1(A^{-1}[x,z]^T,t)}, \ldots, v_{m_s}\sqrt{\hat{a}{_s}_{m_s}(A^{-1}[x,z]^T,t)}\right],\\
	&\sigma_z(x,z,t, \epsilon) =  \left[\hspace{-0.5em}\begin{array}{l} \epsilon A_z\left [v_1\sqrt{\hat{a}{_s}_1(A^{-1}[x,z]^T,t)}, \ldots, v_{m_s}\sqrt{\hat{a}{_s}_{m_s}(A^{-1}[x,z]^T,t)}\right]  \\  A_z \bigg [v_{m_s +1} \sqrt{ \epsilon \hat{a}{_f}_1(A^{-1}[x,z]^T,t)}, \ldots, \\ \hspace{11.5em}v_{{m_s +m_f}}\sqrt{ \epsilon \hat{a}{_f}_{m_f}(A^{-1}[x,z]^T,t)}\bigg] \end{array} \hspace{-0.75em} \right]^T.
\end{align*}
\normalsize
\begin{proof}
See Appendix A-1.\\
\end{proof}
\end{claim}

Based on the result of Claim 1, in this work we consider the Linear Noise Approximation represented in the singular perturbation form: 
\begin{align}
\label{sys_ori1}		\dot{x} &= f_x(x,z,t), \hspace{10em} x(0) = x_0,\\
\label{sys_ori2}		\epsilon \dot{z} &= f_z(x,z,t,\epsilon),  \hspace{9.4em} z(0) = z_0,	\\
\label{sys_ori3}		\dot{\psi_x} &= A_{1}(x,z,t) \psi_x + A_2(x,z,t) \psi_z + \sigma_x(x,z,t) \Gamma_x,  \qquad  \psi_x(0) = {\psi_x}_0,\\
\label{sys_ori4}	 	\epsilon \dot{\psi_z} &= B_1(x,z,t,\epsilon) \psi_x + B_2(x,z,t,\epsilon) \psi_z + \sigma_z(x,z,t, \epsilon) \Gamma_z,	\ \psi_z(0) = {\psi_z}_0,
\end{align}
where 
$x \in D_x \subset \mathbb{R}^{n_s}$, $\psi_x \in D_{\psi_x} \subset \mathbb{R}^{n_s}$ are the slow variables and  $z \in D_z \subset \mathbb{R}^{n_f}$, $\psi_z \in D_{\psi_z} \subset \mathbb{R}^{n_f}$ are the fast variables.  $\Gamma_x$ is an $m_s$-dimensional white noise process. Then, $\Gamma_z = [\Gamma_x, \Gamma_f]^T, $ where $\Gamma_f$ is an $m_f$-dimensional white noise process. 

We refer to the system (\ref{sys_ori1})  - (\ref{sys_ori4}) as the original system and obtain a reduced order model when $\epsilon = 0$. To this end, we make the following assumptions on system (\ref{sys_ori1})  - (\ref{sys_ori4}) for $x \in D_x \subset \mathbb{R}^{n_s}$, $z \in D_z \subset \mathbb{R}^{n_f}$ and $t \in [0, t_1]$.

\begin{assumption}
\label{a1} \rm{ The functions $f_x(x, z, t)$, $f_z(x, z, t, \epsilon)$ are twice continuously differentiable. The Jacobian $\frac{\partial f_z(x, z, t, 0) }{\partial z}$ has continuous first and second partial derivatives with respect to its arguments.}
\end{assumption}

\begin{assumption}
\label{a2} \rm{The matrix-valued functions $\sigma_x(x, z, t)\sigma_x(x, z, t)^T$, $\sigma_z(x, z, t, \epsilon)[\sigma_x(x, z, t) \ 0]^T$ and $\sigma_z(x, z, t, \epsilon)\sigma_z(x, z, t, \epsilon)^T$ are continuously differentiable. Furthermore, we have that $\sigma_z(x, z, t, 0) = 0$ and $\lim_{\epsilon \to 0} \frac{\sigma_z(x,z,t,\epsilon)\sigma_z(x,z,t,\epsilon)^T}{\epsilon} = \sigma(x,z,t)$ where $\sigma(x,z,t)$ is bounded for given $x, z, t$ and $\frac{\partial \sigma(x,z,t)}{\partial z}$ is continuous.}
\end{assumption}

\begin{assumption}
\label{a3}
\rm{There exists an isolated real root $z = \gamma_1(x,t)$, for the equation $f_z(x,z,t,0) = 0$, for which, the matrix $\frac{\partial f_z(x, z, t, 0)}{\partial z} \big|_{z=\gamma_1(x,t)}$ is Hurwitz, uniformly in $x$ and $t$. Furthermore, we have that the first partial derivative of $\gamma_1(x,t)$ is continuous with respect to its arguments. Also, the initial condition $z_0$ is in the region of attraction of the equilibrium point $z = \gamma_1(x_0,0)$ for the system $\frac{dz}{d \tau} = f_z(x_0,z,0,0)$.}
\end{assumption}

\begin{assumption}
\label{a4}
\rm{The system $\dot{\bar{x}} = f_x(\bar{x},\gamma_1(\bar{x},t),t)$ has a unique solution $\bar{x} \in S$ where $S$ is a compact subset of $D_x$ for $t \in [0, t_1]$.}
\end{assumption}

\section{Preliminary Results}
\label{pre_results}
\subsection{Reduced System}

The reduced system is defined by setting $\epsilon = 0$ in the original system (\ref{sys_ori1})  - (\ref{sys_ori4}), which yields 
\begin{align}
\label{sys_red_deri1}		f_z(x,z,t,0) &= 0 ,\\
\label{sys_red_deri2}		B_1(x,z,t,0) \psi_x + B_2(x,z,t,\epsilon) \psi_z &= 0.	
\end{align}

Let $z = \gamma_1({x},t)$ be an isolated root of equation (\ref{sys_red_deri1}), which satisfies Assumption \ref{a3}. Then, we have that $\psi_z  = -B_2({x},\gamma_1({x},t),t,0)^{-1}B_1({x},\gamma_1({x},t),t,0) \psi_x $ is the unique solution of equation (\ref{sys_red_deri2}). Let $ \gamma_2({x}, t) = -B_2({x},\gamma_1({x},t),t,0)^{-1}B_1({x},\gamma_1({x},t),t,0)$. Then, substituting $z = \gamma_1({x},t)$ and $\psi_z = \gamma_2({x}, t)\psi_x$ in equations (\ref{sys_ori1}) and (\ref{sys_ori3}), we obtain the reduced system 
\begin{align}
\label{sys_red1}		\dot{\bar{x}} &= f_x(\bar{x},\gamma_1(\bar{x},t),t), \hspace{8.1em} \bar{x}(0) = x_0,\\
\label{sys_red2}		\hspace{-2em} \dot{\bar{\psi}}_x &= {A}(\bar{x},t)\bar{\psi}_x + \sigma_x(\bar{x},\gamma_1(\bar{x},t),t) \Gamma_x, \hspace{1em} \bar{\psi}_x(0) = {\psi_x}_0,
\end{align}
where\\
$ {A}(\bar{x},t) $ $=$ $ A_1(\bar{x},\gamma_1(\bar{x},t),t)\bar{\psi}_x $ $+$ $ A_2(\bar{x},\gamma_1(\bar{x},t),t)\gamma_2(\bar{x}, t)$.

Next, we derive the first and second moment dynamics of the variable $\bar{\psi}_x$ in the reduced system. To this end, we make the following claim:

\begin{claim}
The first and second moment dynamics for the variable $\bar{\psi}_x$ of the reduced system (\ref{sys_red1}) - (\ref{sys_red2}) can be written in the form 
\begin{align}
\label{mo_red1}		 \frac{d \mathbb{E}[\bar{\psi}_x]}{dt}  &= {A}(\bar{x},t) \mathbb{E}[\bar{\psi}_x] , \hspace{3.5em} \mathbb{E}[\bar{\psi}_x(0)] = {\psi_x}_0, \\
\notag	 \frac{d \mathbb{E}[\bar{\psi}_x\bar{\psi}_x^T] }{dt}  &= {A}(\bar{x},t) \mathbb{E}[\bar{\psi}_x\bar{\psi}_x^T]   +  \mathbb{E}[\bar{\psi}_x\bar{\psi}_x^T] {A}(\bar{x},t)^T + \sigma_x(\bar{x},\gamma_1(\bar{x},t),t) ,t)\sigma_x(\bar{x},\gamma_1(\bar{x},t),t),t)^T,  \\& \label{mo_red2}  \hspace{18.5em}   \mathbb{E}[\bar{\psi}_x(0)\bar{\psi}_x(0)^T] = {\psi_x}_0{\psi_x}_0^T.
\end{align}
\end{claim}

\begin{proof}
Similar to \cite{bence}, the first and second moment dynamics of $\bar{\psi}_x$ in (\ref{sys_red2}) can be written as 
\begin{align*}
	 \frac{d \mathbb{E}[\bar{\psi}_x]}{dt}  &= \mathbb{E}[{A}(\bar{x},t)\bar{\psi}_x],  \\
\notag	 \frac{d \mathbb{E}[\bar{\psi}_x\bar{\psi}_x^T] }{dt}  &= \mathbb{E}[{A}(\bar{x},t) \bar{\psi}_x\bar{\psi}_x^T] +  \mathbb{E}[\bar{\psi}_x(\bar{\psi}_x^T  {A}(\bar{x},t)^T)]  + \sigma_x(\bar{x},\gamma_1(\bar{x},t),t) ,t)\sigma_x(\bar{x},\gamma_1(\bar{x},t),t),t)^T.
\end{align*}
Since the dynamics of $\bar{x}$ given by (\ref{sys_red1}) are deterministic, using the linearity of the expectation operator we can write the moment dynamics of the reduced system as (\ref{mo_red1}) - (\ref{mo_red2}). 
\end{proof}




Next, we proceed to derive the moment dynamics for $\psi_x$ and $\psi_z$ in the original system (\ref{sys_ori1})  - (\ref{sys_ori4}) given by the following claim.

\begin{claim}
\label{Le_sp}
The first and second moment dynamics for the variables $\psi_x$ and $\psi_z$ of the original system (\ref{sys_ori1})  - (\ref{sys_ori4}) can  be written in the form
\begin{align}
\label{mo_ori1}	&\frac{d \mathbb{E}[\psi_{x}]}{dt}  =  A_1(x,z,t) \mathbb{E}[\psi_x] + A_2(x,z,t)\mathbb{E}[\psi_z], \\ 
\notag		&\frac{d\mathbb{E}[\psi_x \psi_x^T]}{dt}  = A_1(x,z,t) \mathbb{E}[\psi_x \psi_x^T] + A_2(x,z,t)\mathbb{E}[\psi_z\psi_x^T]  + \mathbb{E}[\psi_x \psi_x^T] A_1(x,z,t)^T \\&  + (\mathbb{E}[\psi_z\psi_x^T])^T A_2(x,z,t)^T  \label{mo_ori_mid1}  +  \sigma_x(x,z,t)\sigma_x(x,z,t)^T, \\ 
\label{mo_ori_mid2}	&\epsilon\frac{d \mathbb{E}[\psi_{z}]}{dt}  =  B_1(x,z,t,\epsilon) \mathbb{E}[\psi_x] + B_2(x,z,t,\epsilon)\mathbb{E}[\psi_z], \\
\notag		&\epsilon\frac{d\mathbb{E}[\psi_z \psi_x^T]}{dt}  =  \epsilon\mathbb{E}[\psi_z \psi_x^T] A_1(x,z,t)^T  \notag + \epsilon\mathbb{E}[\psi_z\psi_z^T] A_2(x,z,t)^T  + B_1(x,z,t,\epsilon) \mathbb{E}[\psi_x \psi_x^T] \\&   \label{mo_ori_mid3}  + B_2(x,z,t,\epsilon)\mathbb{E}[\psi_z \psi_x^T]  + \sigma_z(x,z,t,\epsilon)[\sigma_x(x,z,t) \ 0]^T, \\ 
\notag	&\epsilon\frac{d\mathbb{E}[\psi_z \psi_z^T]}{dt}  = B_1(x,z,t,\epsilon) \mathbb{E}[\psi_x \psi_z^T]   + B_2(x,z,t,\epsilon) \mathbb{E}[\psi_z\psi_z^T]  + \mathbb{E}[\psi_z \psi_x^T] B_1(x,z,t,\epsilon)^T \\&  \label{mo_ori2}  + \mathbb{E}[\psi_z\psi_z^T] B_2(x,z,t,\epsilon)^T  +  \frac{1}{\epsilon}\sigma_z(x,z,t,\epsilon)\sigma_z(x,z,t,\epsilon)^T, 
\end{align}
\normalsize
where $x$ and $z$ are the solutions of the equations (\ref{sys_ori1})  - (\ref{sys_ori2}), and the initial conditions are given by $ \mathbb{E}[{\psi}_x(0)] = {\psi_x}_0,$ $\mathbb{E}[\psi_x \psi_x^T(0)] = {\psi_x}_0{\psi_x}_0^T$, $\mathbb{E}[{\psi}_z(0)] = {\psi_z}_0$, $ \mathbb{E}[\psi_z\psi_x^T(0)] = {\psi_z}_0{\psi_x}_0^T,$ $\mathbb{E}[\psi_z\psi_z^T(0)] = {\psi_z}_0{\psi_z}_0^T$.
\end{claim}

\begin{proof}
The equations (\ref{sys_ori3})  - (\ref{sys_ori4}) can be written in the form 
\begin{align*}
		\dot{\psi}_x &= A_1(x,z,t) \psi_x + A_2(x,z,t) \psi_z + [\sigma_x(x,z,t) \ 0] \Gamma_z,\\
		\epsilon \dot{\psi}_z &= B_1(x,z,t,\epsilon) \psi_x + B_2(x,z,t,\epsilon) \psi_z +\sigma_z(x,z,t,\epsilon) \Gamma_z,
\end{align*}
where $[\begin{array}{cc}\sigma_x(x,z,t) & 0\end{array}] \in \mathbb{R}^{n \times (m_s + m_f)}$. 
Then, using the fact that the $x$ and $z$ are deterministic and the linearity of the expectation operator, the dynamics for the first moments can be written as 
\begin{align}
\label{mo1}	\frac{d \mathbb{E}[\psi_{x}]}{dt}  &=  A_1(x,z,t) \mathbb{E}[\psi_x] + A_2(x,z,t)\mathbb{E}[\psi_z] ,\\
 			\frac{d \mathbb{E}[\psi_{z}]}{dt}  &=  \frac{1}{\epsilon}B_1(x,z,t,\epsilon) \mathbb{E}[\psi_x] + \frac{1}{\epsilon}B_2(x,z,t,\epsilon)\mathbb{E}[\psi_z].
\end{align}
Similarly, using Proposition III.1 in \cite{bence}, the second moment dynamics can be written as  
\begin{align}
\notag 	& \frac{d}{dt} \mathbb{E} \left[\begin{array}{cc} \psi_x\psi_x^T & \psi_x\psi_z^T \\  \psi_z \psi_x^T & \psi_z\psi_z^T\end{array}\right] =\\& 
\notag \bigg[ \arraycolsep=3.5pt \def\arraystretch{1.2} \begin{array}{c} \psi_x(A_1(x,z,t) \psi_x + A_2(x,z,t) \psi_z)^T \\ 
\psi_z(A_1(x,z,t) \psi_x + A_2(x,z,t) \psi_z)^T \end{array}   \arraycolsep=3.5pt \def\arraystretch{1.2} \begin{array}{c}
 \frac{1}{\epsilon} \psi_x(B_1(x,z,t,\epsilon) \psi_x + B_2(x,z,t,\epsilon) \psi_z)^T\\   \frac{1}{\epsilon}\psi_z(B_1(x,z,t,\epsilon) \psi_x + B_2(x,z,t,\epsilon) \psi_z)^T \end{array} \bigg] \\& \notag +  \bigg[\arraycolsep=3.5pt \def
\arraystretch{1.2}  \begin{array}{c} (A_1(x,z,t) \psi_x + A_2(x,z,t) \psi_z)\psi_x^T  \\ \frac{1}{\epsilon}(B_1(x,z,t,\epsilon) \psi_x + B_2(x,z,t,\epsilon) \psi_z)\psi_x^T  \end{array}    \arraycolsep=3.5pt \def\arraystretch{1.2} \begin{array}{c} (A_1(x,z,t) \psi_x + A_2(x,z,t) \psi_z)\psi_z^T \\  \frac{1}{\epsilon}(B_1(x,z,t,\epsilon) \psi_x + B_2(x,z,t,\epsilon) \psi_z)\psi_z^T \end{array} \bigg] \\& \label{sec_mo} 
+ \bigg[ \arraycolsep=3.5pt \def\arraystretch{1.2} \begin{array}{c} \sigma_x(x,z,t)\sigma_x(x,z,t)^T \\  \frac{1}{{\epsilon}}\sigma_z(x,z,t,\epsilon)[\begin{array}{cc}\sigma_x(x,z,t) & 0\end{array}]^T\end{array} 
\arraycolsep=3.5pt \def\arraystretch{1.2}  \begin{array}{c} \frac{1}{{\epsilon}}[\begin{array}{cc}\sigma_x(x,z,t) & 0\end{array}]\sigma_z(x,z,t,\epsilon)^T \\  \frac{1}{\epsilon^2}\sigma_z(x,z,t,\epsilon)\sigma_z(x,z,t,\epsilon)^T \end{array}\bigg]. 
\end{align}
Employing the linearity of the expectation operator, we can sum the corresponding entries of the matrices in equation (\ref{sec_mo}), and multiply by $\epsilon$ to write the moment equations (\ref{mo1}) - (\ref{sec_mo}) in the form of (\ref{mo_ori1}) - (\ref{mo_ori2}). Note that, since $\mathbb{E}[\psi_x\psi_z^T] = (\mathbb{E}[\psi_z\psi_x^T])^T$, we have eliminated the dynamics of the variable  $\mathbb{E}[\psi_x\psi_z^T]$.
\end{proof}



\begin{claim}
\label{claim_mo_ori}
Setting $\epsilon = 0$ in the system of moment dynamics (\ref{mo_ori1}) - (\ref{mo_ori2}) and the dynamics of $x$ and $z$ given by (\ref{sys_ori1}) - (\ref{sys_ori2}), yields the moment dynamics of the reduced system (\ref{mo_red1}) - (\ref{mo_red2}) where the dynamics of $\bar{x}$ are given by (\ref{sys_red1}). 

\end{claim}

\begin{proof}
Setting $\epsilon = 0$ in equations (\ref{sys_ori1}) - (\ref{sys_ori2}) and (\ref{mo_ori_mid2}) - (\ref{mo_ori_mid3}), yields 
\begin{align}
\label{z_0} 	0 &= f_z(x,z,t,0),\\
\label{psi_x0}	0 &=  B_1(x,z,t,0) \mathbb{E}[\psi_x] + B_2(x,z,t,0)\mathbb{E}[\psi_z] , \\
\label{psi_z0}	0 &= B_1(x,z,t,0) \mathbb{E}[\psi_x \psi_x^T] + B_2(x,z,t,0)\mathbb{E}[\psi_z \psi_x^T].
\end{align}
By definition of the reduced system, we have that $z = \gamma_1(x,t)$ is an isolated root for equation (\ref{z_0}). Then, under  Assumption \ref{a3}, we have that the unique solutions for the equations (\ref{psi_x0}) and (\ref{psi_z0})
are given by 
\begin{align}
\notag	\mathbb{E}[\psi_z]  &= -B_2(x, \gamma_1(x,t),t,0)^{-1}(B_1(x, \gamma_1(x,t),t,0)\mathbb{E}[\psi_x]) \\&  \label{sol_psi_z} =  \gamma_2(x,t)\mathbb{E}[\psi_x], \\
\notag	\mathbb{E}[\psi_z \psi_x^T] &= -B_2(x, \gamma_1(x,t),t,0)^{-1}(B_1(x, \gamma_1(x,t),t,0) \mathbb{E}[\psi_x \psi_x^T])  \\& \label{sol_psi_zx} = \gamma_2(x,t)\mathbb{E}[\psi_x \psi_x^T].
\end{align}
Substituting $z = \gamma_1({x},t)$ and equations (\ref{sol_psi_z}) - (\ref{sol_psi_zx}), in (\ref{sys_ori1}) and (\ref{mo_ori1}) - (\ref{mo_ori2}) results in
\begin{align}
	\label{sys_ori1_pr}		\dot{x} &= f_x(x,\gamma_1(x,t),t),\\
\label{mo_ori1_pr}		\frac{d \mathbb{E}[\psi_{x}]}{dt}  &=  A_1(x,\gamma_1({x},t),t) \mathbb{E}[\psi_x]  + A_2(x,\gamma_1({x},t),t)\gamma_2(x,t)\mathbb{E}[\psi_x], \\
\notag		\frac{d\mathbb{E}[\psi_x \psi_x^T]}{dt}  &= A_1(x,\gamma_1({x},t),t) \mathbb{E}[\psi_x \psi_x^T]  + A_2(x,\gamma_1({x},t),t) \gamma_2(x,t)\mathbb{E}[\psi_x \psi_x^T] \\& + \mathbb{E}[\psi_x \psi_x^T] A_1(x,z,t)^T + (\gamma_2(x,t)\mathbb{E}[\psi_x \psi_x^T])^T A_2(x,\gamma_1({x},t),t)^T \\& \label{mo_ori_mid1_pr}  +  \sigma_x(x,\gamma_1({x},t),t)\sigma_x(x,\gamma_1({x},t),t)^T.
\end{align}
It follows that equation (\ref{sys_ori1_pr}) is equivalent to the reduced system given by (\ref{sys_red1}) and since we have that $ {A}({x},t) = A_1({x},\gamma_1({x},t),t)\bar{\psi}_x + A_2({x},\gamma_1({x},t),t)\gamma_2({x}, t)$, the system (\ref{mo_ori1_pr}) - (\ref{mo_ori_mid1_pr}) is equivalent to the moment dynamics of the reduced system given by (\ref{mo_red1}) - (\ref{mo_red2}).
\end{proof}

\section{Main Results}
\label{main_results}
\begin{lemma}
Consider the original system in (\ref{sys_ori1}) - (\ref{sys_ori4}), the reduced system in (\ref{sys_red1}) - (\ref{sys_red2}), and the moment dynamics for the original and reduced systems in (\ref{mo_ori1}) - (\ref{mo_ori2}), (\ref{mo_red1}) - (\ref{mo_red2}) respectively. We have that, under Assumptions 1 - 3,  the commutative diagram in Fig. \ref{comm} holds.
\begin{figure*}[t!] 
\hspace{0.5em}
\vspace{-1em}
\def\svgwidth{450pt} 
\footnotesize
\centering
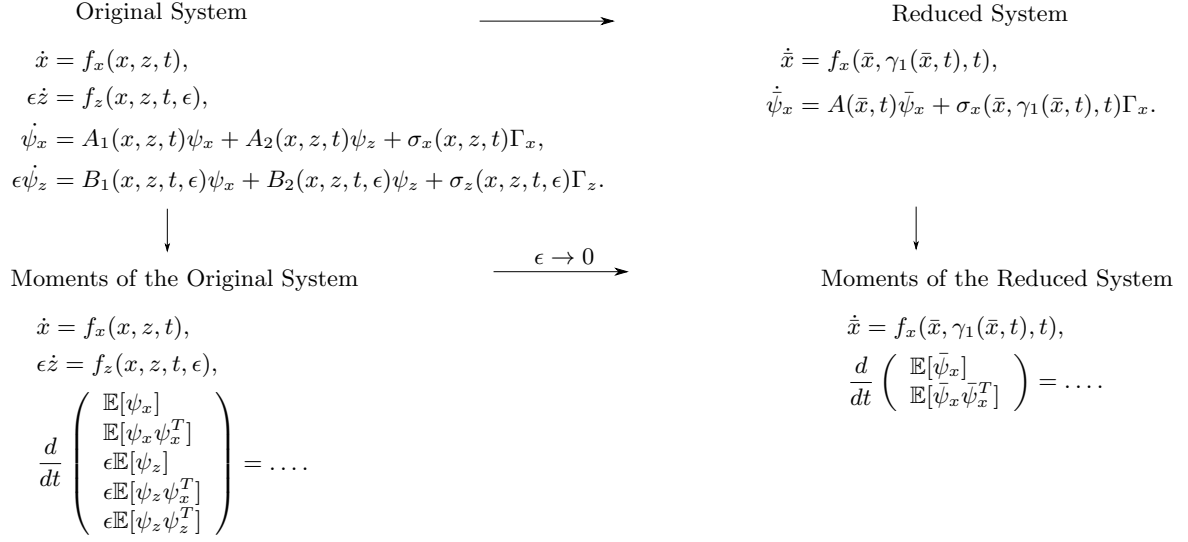\vspace{14em}
\vspace{-8em}
\caption{\small{Commutative Diagram.}}
\label{comm}
\end{figure*}  
\end{lemma}

\begin{proof}
The proof follows from Claim 1, Claim 2 and Claim 3.
\end{proof}

\begin{theorem}
Consider the original system (\ref{sys_ori1}) - (\ref{sys_ori4}), the reduced system in (\ref{sys_red1}) -  (\ref{sys_red2}) and the moment dynamics for the original and reduced systems in (\ref{mo_ori1}) - (\ref{mo_ori2}), (\ref{mo_red1}) - (\ref{mo_red2}) respectively. Then, under Assumptions \ref{a1} - \ref{a4}, there exists  $\epsilon^* \ge 0$ such that for $0 < \epsilon < \epsilon^*$, we have 
\begin{align}
\label{thm1_res1}	&\| x(t) - \bar{x}(t) \|= O(\epsilon),  \ t \in [0, t_1], \\
\label{thm1_res2}	&\| \mathbb{E}[\psi_x(t)] - \mathbb{E}[\bar{\psi}_x(t)] \| = O(\epsilon), \\
\label{thm1_res2}	&\|  \mathbb{E}[{\psi}_x(t){\psi}_x(t)^T]  - \mathbb{E}[\bar{\psi}_x(t)\bar{\psi}_x(t)^T] \| = O(\epsilon).
\end{align}
\end{theorem}

\begin{proof}
From Lemma 1, we see that setting $\epsilon = 0$ in the moment dynamics of the original system  (\ref{mo_ori1}) - (\ref{mo_ori2}) and in the dynamics of $x$ and $z$ given by (\ref{sys_ori1}) - (\ref{sys_ori2}), yields the moment dynamics of the reduced system  (\ref{mo_red1})  - (\ref{mo_red2}) where the dynamics of $\bar{x}$ are given by (\ref{sys_red1}). Therefore to prove Theorem 1, we apply Tikhonov's theorem \cite{khalil} to the system of moment dynamics of the original system given by (\ref{mo_ori1}) - (\ref{mo_ori2}) together with the dynamics of $x$ and $z$ given by (\ref{sys_ori1}) - (\ref{sys_ori2}). In order to apply Tikhonov's theorem, we first prove that the assumptions of the Tikhonov's theorem are satisfied. To this end, let us define the boundary layer variables 
\begin{align}
	 	b_1 &= z - \gamma_1(x,t), \\
		b_2 &= \mathbb{E}[\psi_z] -  \gamma_2(x,t)\mathbb{E}[\psi_x], \\
 		b_3 &= \mathbb{E}[\psi_z \psi_x^T] -  \gamma_2(x,t)\mathbb{E}[\psi_x \psi_x^T]. 
\end{align}
The dynamics of the boundary layer variables are given by 
\begin{align*}
	\frac{db_1}{dt} &=  \frac{dz}{dt} - \frac{d\gamma_1(x,t)}{dt},\\
	\frac{db_{2}}{dt} &= \frac{d\mathbb{E}[{\psi_z}]}{dt} - \frac{d\gamma_2(x,t)\mathbb{E}[\psi_x]}{dt},\\
	\frac{db_{3}}{dt} &= \frac{d\mathbb{E}[{\psi_z}\psi_x]}{dt} - \frac{d\gamma_2(x,t)\mathbb{E}[\psi_x \psi_x^T]}{dt}.
\end{align*}
Denote by $\tau = t/\epsilon$ the time variable in the fast time-scale. Then, expanding using the chain rule, we have
\begin{align*}
	\frac{db_1}{d\tau} &=  \epsilon\frac{dz}{dt} -  \epsilon\frac{\partial \gamma_1(x,t)}{\partial t} -  \epsilon\frac{\partial \gamma_1(x,t)}{\partial x} \frac{dx}{dt},\\
	\frac{db_{2}}{d\tau} &= \epsilon\frac{d\mathbb{E}[{\psi_z}]}{dt} -   \epsilon\mathbb{E}[\psi_x] \frac{\partial \gamma_2(x,t)}{\partial t} -   \epsilon\mathbb{E}[\psi_x ] \frac{\partial \gamma_2(x,t)}{\partial x}\frac{dx}{dt}  -   \epsilon \gamma_2(x,t) \frac{d \mathbb{E}[\psi_x ] }{d t} ,\\
	\frac{db_{3}}{d\tau} &= \epsilon\frac{d\mathbb{E}[{\psi_z}\psi_x^T]}{dt} -  \epsilon \mathbb{E}[\psi_x \psi_x^T] \frac{\partial \gamma_2(x,t)}{\partial t}  - \epsilon\mathbb{E}[\psi_x \psi_x^T] \frac{\partial \gamma_2(x,t)}{\partial x}\frac{dx}{dt} -  \epsilon\gamma_2(x,t) \frac{d \mathbb{E}[\psi_x \psi_x^T] }{d t} .	
\end{align*}
Substituting from equations (\ref{sys_ori2}), (\ref{mo_ori_mid2}) and (\ref{mo_ori2}) yields 
\begin{align}
\label{boundary_proof1}	\frac{db_1}{d\tau} &=  f_z(x,z,t,\epsilon) -  \epsilon\frac{\partial \gamma_1(x,t)}{\partial t} -  \epsilon\frac{\partial \gamma_1(x,t)}{\partial x} \frac{dx}{dt},\\
\notag	\frac{db_{2}}{d\tau} &=  B_1(x,z,t,\epsilon) \mathbb{E}[\psi_x] + B_2(x,z,t,\epsilon)\mathbb{E}[\psi_z]  \\& \notag -  \epsilon \mathbb{E}[\psi_x] \frac{\partial \gamma_2(x,t)}{\partial t} -  \epsilon\mathbb{E}[\psi_x ] \frac{\partial \gamma_2(x,t)}{\partial x}\frac{dx}{dt}  -   \epsilon \gamma_2(x,t) \frac{d \mathbb{E}[\psi_x ] }{d t} ,\\
\notag	\frac{db_{3}}{d\tau} &= \epsilon\mathbb{E}[\psi_z \psi_x^T] A_1(x,z,t)^T + \epsilon\mathbb{E}[\psi_z\psi_z^T] A_2(x,z,t)^T  + B_1(x,z,t,\epsilon) \mathbb{E}[\psi_x \psi_x^T] \\& \notag + B_2(x,z,t,\epsilon)\mathbb{E}[\psi_z \psi_x^T]  + \sigma_z(x,z,t, \epsilon)[\sigma_x(x,z,t) \ 0]^T -  \epsilon\mathbb{E}[\psi_x \psi_x^T] \frac{\partial \gamma_2(x,t)}{\partial t} \\&\label{boundary_proof2} - \epsilon \mathbb{E}[\psi_x \psi_x^T] \frac{\partial \gamma_2(x,t)}{\partial x}\frac{dx}{dt}  -  \epsilon \gamma_2(x,t) \frac{d \mathbb{E}[\psi_x \psi_x^T] }{d t}.	
\end{align}
where we take $z = b_1 + \gamma_1(x,t)$ and $\mathbb{E}[\psi_z] = b_2 +  \gamma_2(x,t)\mathbb{E}[\psi_x]$, and $\mathbb{E}[\psi_z \psi_x^T] = b_3 + \gamma_2(x,t) \mathbb{E}[\psi_x \psi_x^T]$. Since, from Assumption $\ref{a3}$, $\gamma_1(x,t)$ is a continuously differentiable functions in its arguments, we have that $\frac{\partial \gamma_1(x,t)}{\partial t}$, $\frac{\partial \gamma_1(x,t)}{dx}$ are bounded in a finite time interval $t \in [0, t_1]$. Since $\gamma_2({x}, t) = -B_2({x},\gamma_1({x},t),t,0)^{-1}B_1({x},\gamma_1({x},t),t,0)$, and $B_1$ and $B_2$ are continuously differentiable from Assumption \ref{a1}, we have that $\frac{\partial \gamma_{2}(x, t)}{\partial x}$ and $\frac{\partial  \gamma_{2}(x, t)}{\partial t}$  are bounded in a finite time interval $t \in [0, t_1]$.
Then, the boundary layer system obtained by setting $\epsilon = 0$ in (\ref{boundary_proof1}) - (\ref{boundary_proof2}) is given by 
\begin{align}
\label{boundary_1}	\frac{db_1}{d\tau} &= f_z(x, b_1 + \gamma_1(x,t), t, 0),\\
\notag \frac{db_{2}}{d\tau} &=  B_1(x, b_1 + \gamma_1(x,t),t,0) \mathbb{E}[\psi_x] + B_2(x, b_1 + \gamma_1(x,t),t,0)( b_2 +  \gamma_2(x,t)\mathbb{E}[\psi_x])  \\& \label{boundary_3} =: g_1(b_1,b_2, x,t),\\
\notag \frac{db_{3}}{d\tau} &=  B_1(x, b_1 + \gamma_1(x,t),t,0) \mathbb{E}[\psi_x \psi_x^T] + B_2(x, b_1 + \gamma_1(x,t),t,0)(b_3 + \gamma_2(x,t) \mathbb{E}[\psi_x \psi_x^T]) \\& \label{boundary_2} =:  g_2(b_1,b_3, x,t).
\end{align}
To prove that the origin of the boundary layer system is exponentially stable, we consider the dynamics of the vectors $e_i = [b_1, b_2, b_{3i}]$ where $b_{3i}$ represent the columns of the matrix $b_3$ for $i = 1, \ldots, n$. Similarly, denote the columns of the matrix $g_{2}(b_1,b_3, x,t)$ by $g_{2i}(b_1,b_3, x,t)$ representing the dynamics for each $b_{3i}$.  Linearizing the system (\ref{boundary_1}) - (\ref{boundary_2}) around the origin,  we obtain the dynamics for $\tilde{e}_i = e_i - 0$ as

\begin{align}
	\frac{d\tilde{e}_i}{d\tau} &= \left[\begin{array}{ccc} J_{11} & 0 & 0 \\  J_{21} & J_{22} & 0\\ J_{31} & 0 & J_{33}\end{array} \right] \tilde{e}_i, \  \qquad  i = \{ 1, \ldots, n \},
\end{align}
where $J_{11} = \frac{\partial f_z(x, b_1 + \gamma_1(x,t), t, 0)}{\partial b_1}\big|_{b_1=0}$, $J_{21} = \frac{\partial g_1(b_1,b_2,x,t)}{\partial b_1}\big|_{e_i=0} $, $J_{22} = B_2(x, b_1 + \gamma_1(x,t) , t, 0)\big|_{b_1=0} $, $J_{31} = \frac{\partial g_{2i}(b_1,b_3, x,t)}{\partial b_1}\big|_{e_i=0}$, and $J_{33} = B_2(x, b_1 + \gamma_1(x,t) , t, 0)\big|_{b_1=0}$.
Since the eigenvalues of a block triangular matrix are given by the union of eigenvalues of the diagonal blocks, we consider the eigenvalues of $\frac{\partial f_z(x, b_1 + \gamma_1(x,t), t, 0)}{\partial b_1}\big|_{b_1=0}$ and $ B_2(x, b_1 + \gamma_1(x,t) , t, 0)\big|_{b_1=0}$.  Under Assumption \ref{a3}, we have that the matrix $\frac{\partial f_z(x, b_1 + \gamma_1(x,t), t, 0)}{\partial b_1}\big|_{b_1=0} = \frac{\partial f_z(x, z, t, 0)}{\partial z} \frac{d z}{d b_1}\big|_{z = \gamma_1(x,t)} = \frac{\partial f_z(x, z, t, 0)}{\partial z} \big|_{z = \gamma_1(x,t)} $ is Hurwitz. From the definition of the original system (\ref{sys_ori1}) - (\ref{sys_ori4}), we have $B_2(x, z , t, \epsilon) = \frac{\partial f_z(x, z, t, \epsilon)}{\partial z} $. Therefore, $ B_2(x, b_1 + \gamma_1(x,t) , t, 0)\big|_{b_1=0} = \frac{\partial f_z(x, z, t, 0)}{\partial z} \big|_{z=\gamma_1(x,t)}$, which is Hurwitz under Assumption \ref{a3}. Thus, the boundary layer system is exponentially stable. 
 
From Assumptions \ref{a1} and \ref{a2} we have that the functions $f_x(x,z,t)$,  $f_z(x,z,t,\epsilon)$, $A_1(x,z,t)$, $A_2(x,z,t)$, $B_1(x,z,t,\epsilon)$, $B_2(x,z,t,\epsilon)$,  $\sigma_x(x, z, t)\sigma_x(x, z, t)^T$, $\sigma_z(x, z, t, \epsilon)[\sigma_x(x, z, t) \ 0]^T$ and $\sigma_z(x, z, t, \epsilon)\sigma_z(x, z, t, \epsilon)^T$ and their first partial derivatives are continuously differentiable. From Assumption \ref{a1} we have that the $\frac{\partial  f_z(x,z,t,0)}{\partial z}$, $\frac{\partial  B_1(x,z,t,0)}{\partial z}$, $\frac{\partial  B_2(x,z,t,0)}{\partial z}$ have continuous first partial derivatives with respect to their arguments. From Assumptions \ref{a1} and \ref{a3} we have that the $\gamma_1(x,t)$, $\gamma_2(x,t)\mathbb{E}[\psi_x]$, $\gamma_2(x,t)\mathbb{E}[\psi_x \psi_x^T]$ have continuous first partial derivatives with respect to their arguments.  From Assumption \ref{a4} we have that the reduced system (\ref{sys_red1}) has a unique bounded solution for $t \in [0, t_1]$.  Since the moment equations (\ref{mo_red1})  - (\ref{mo_red2}) are linear in $\mathbb{E}[\bar{\psi}_x]$ and $\mathbb{E}[\bar{\psi}_x\bar{\psi}_x^T]$ there exists a unique solution to (\ref{mo_red1})  - (\ref{mo_red2}) for $t \in [0, t_1]$.  From Assumption \ref{a3} we have that the initial condition $z_0$ is in the region of attraction of the equilibrium point $\gamma_1(x_0, 0)$, and thus the initial condition $z_0 - \gamma_1(x_0,0)$ for the boundary layer system $b_1$ with the frozen variables $x = x_0$, $t = 0$, is in the region of attraction of the equilibrium point $b_1 = 0$. Then, since the system (\ref{boundary_3}) - (\ref{boundary_2}) is linear in the variables $b_2$ and $b_3$, it follows from Assumption 3 that  $z_0 - \gamma_1(x_0,0)$, ${\psi_z}_0 -  \gamma_2(x_0, 0){\psi_x}_0$,  ${\psi_z}_0 {\psi_x}_0^T -  \gamma_2(x_0, 0){\psi_x}_0 {\psi_x}_0^T$ for the variables $b_1$, $b_2$ and $b_3$ are in the region of attraction of the equilibrium point at the origin. Thus, the assumptions of the Tikhonov's theorem on a finite time-interval \cite{khalil} are satisfied and applying the theorem to the moment dynamics of the original system in  (\ref{mo_ori1}) - (\ref{mo_ori2}) and the dynamics of $x$ and $z$ given by (\ref{sys_ori1}) - (\ref{sys_ori2}), we obtain the result (\ref{thm1_res1}) - (\ref{thm1_res2}).\end{proof}

\textbf{Remark:} From \cite{kampen}, we have that $\psi_x(t)$ and $\bar{\psi_x}(t)$ are multivariate Gaussian processes. Since a  Gaussian distribution is fully characterized by their mean and the covariance, and Theorem 1 gives  
\begin{align}
	lim_{\epsilon \to 0} \mathbb{E}[\psi_x(t)] &=  \mathbb{E}[\bar{\psi}_x(t)],\\
	lim_{\epsilon \to 0} \mathbb{E}[\psi_x(t)\psi_x(t)^T] &=  \mathbb{E}[\bar{\psi}_x(t)\bar{\psi}_x(t)^T],
\end{align}
we have that for given $t \in [0, t_1]$, the vector $\psi_x(t)$ converges in distribution to the vector $\bar{\psi}_x(t)$ as $\epsilon \to 0$.

\section{Example}
\label{examples}
In this section we demonstrate the application of the model reduction approach on an example of a biolomelcular system. Consider the system in Fig. \ref{example}, where a phosphorylated protein $\rm{X}^*$ binds to a downstream promoter site p which produces the protein G. 
Such a setup can be seen commonly occurring in natural biological systems, an example being the two component signaling systems in bacteria \cite{koretke2000evolution}.  Moreover, similar setups are also used in synthetic biology to design biological circuits that are robust to the loading effects that appear due to the presence of downstream components \cite{DDV_MSB, Jayanthi_TAC}.   

\begin{figure}[th] 
\centering
\includegraphics[width= 0.6 \textwidth]{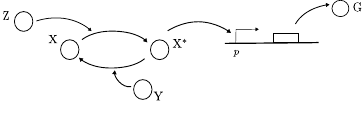}
\vspace{-1em}
\caption{\small{Protein X is phosphorylated by kinase Z and dephosphorylated by phosphatase Y. Phosphorylated protein $\rm{X}^*$ binds to the downstream promoter p.}}
\label{example}
\end{figure}  

The chemical reactions for the system are as follows:
$\mathrm{X} + \mathrm{Z} \xrightarrow{k_1}\mathrm{X^*}+\mathrm{Z}, \ 
\mathrm{X^*} + \mathrm{Y} \xrightarrow{k_2}\mathrm{X}+\mathrm{Y}, \ 
\mathrm{X^*} + \mathrm{p} \xrightleftharpoons[k_{\rm{off}}]{k_{\rm{on}}} \mathrm{C},$
$\mathrm{C} \xrightarrow{\beta} \mathrm{C} + \mathrm{G},$\
$\mathrm{G} \xrightarrow{\delta} \phi $.
The protein X is phosphorylated by kinase Z and dephosphorylated by phosphatase Y with the rate constants $k_1$ and $k_2$, respectively. The binding between phosphorylated protein $\rm{X}^*$ and promoter p produces a complex C, where $k_{\rm{on}}$ and  $k_{\rm{off}}$ are the binding and unbinding rate constants. Protein G is produced at rate $\beta$, which encapsulates both transcription and translation processes and decays at rate $\delta$, which includes both degradation and dilution. We assume that the total concentration of protein X and promoter p are conserved, giving ${X}_{tot} = {X} + {X}^* + {C}$ and ${p}_{tot} = {p} + {C}$, where the lower-case letters denote the corresponding macroscopic concentrations. Then, the dynamics for the macroscopic concentrations of X$^*$, C and G can be written as 

\small 
\begin{align}
\label{ex1}	 \frac{dx^*}{dt} &= k_1Z(t)(X_{tot} - {x^*} - {c}) - k_2Yx^*  - k_{\rm{on}}x^*(p_{tot} - c) + k_{\rm{off}}c, \\
		\frac{dc}{dt} &= k_{\rm{on}}x^*(p_{tot} - c) -  k_{\rm{off}}c,  \\
\label{ex2}	\frac{dg}{dt} &= \beta c - \delta g.
\end{align}
\normalsize
Binding and unbinding reactions are much faster than phosphorylation/dephosphorylation, and therefore, we can write $k_2Y/k_{\rm{off}} = \epsilon \ll 1$. Taking $k_{d} = k_{\rm{off}}/k_{\rm{on}}$, we have 

\small 
\begin{align}
\label{ex3} 	\frac{dx^*}{dt} &= k_1Z(t)(X_{tot} - {x^*} - {c}) - k_2Yx^*   - \frac{k_2Y}{\epsilon k_d}x^*(p_{tot} - c) +  \frac{k_2Y}{\epsilon}c,  \\
	 		\frac{dc}{dt} &= \frac{k_2Y}{\epsilon k_d}x^*(p_{tot} - c) -  \frac{k_2Y}{\epsilon}c,  \\
\label{ex4} 	\frac{dg}{dt} &= \beta c - \delta g.
\end{align}\normalsize
The system (\ref{ex3}) - (\ref{ex4}) is in the form of system (\ref{LNA_intro1}), with $y = [x^*, \ c, \ g ]^T$. To take the system in to the singular perturbation form given in (\ref{sys_ori1}) - (\ref{sys_ori2}), we consider the change of variable $v = x^* + c$, which yields

\small 
\begin{align} 
\label{ex5} 	\frac{dv}{dt} &= k_1Z(t)( X_{tot} - v) - k_2Y(v - c), \\
			\frac{dg}{dt} &= \beta c - \delta g,\\
\label{ex6} 	\epsilon \frac{dc}{dt} &= \frac{k_2Y}{k_d}(v - c)(p_{tot} - c) -  k_2Yc.
\end{align}
\normalsize
This change of coordinates corresponds to having $A_x = [1 \ 1 \ 0, \ 0 \ 0 \ 1]^T$, $A_z = [0 \ 1 \ 0]$, $x = [v, g]^T$ and $z = c$ in Claim 1. Then, the dynamics for the stochastic fluctuations can be written as 

\small 
\begin{align}
\label{ex7} 	\frac{d\psi_{v}}{dt} &= (-k_1Z(t) - k_2Y)\psi_{v}  + k_2Y \psi_c  + \sqrt{k_1Z(t)( X_{tot} - v)} \Gamma_1 - \sqrt{k_2Y(v - c)}\Gamma_2, \\
	\frac{d\psi_g}{dt} &= \beta \psi_c - \delta \psi_g + \sqrt{\beta c} \Gamma_3 - \sqrt{\delta g}\Gamma_4,\\
 \notag	\epsilon \frac{d\psi_c}{dt} &=  \frac{k_2Y}{k_d} (p_{tot} - c) \psi_v  + \left(- \frac{k_2Y}{k_d}v -  \frac{k_2Y}{k_d} p_{tot} + 2  \frac{k_2Y}{k_d} c - k_2 Y\right)\psi_c \\&  \label{ex8} + \sqrt{ \epsilon \frac{k_2Y}{k_d}(v - c)(p_{tot} - c)}\Gamma_5  -  \sqrt{ \epsilon {k_2Y}c} \Gamma_6.
\end{align}
\normalsize
with $\psi_x = [\psi_v, \psi_g]^T$ and $\psi_x = [\psi_v, \psi_g]^T$. Therefore, the equations (\ref{ex5}) - (\ref{ex8}) are in the form of the original system in (\ref{sys_ori1}) - (\ref{sys_ori2}) with $x = [v, g]^T$ and $z = c$. 
It follows that Assumptions \ref{a1} and \ref{a2} are satisfied since the system functions of  (\ref{ex5}) - (\ref{ex6}) are polynomials of the state variables. We evaluate $f_z = \frac{k_2Y}{k_d}(v - {z})(p_{tot} - {z}) -  k_2Y{z} = 0$, which yields the unique solution $z(v) = \frac{1}{2}(v + p_{tot} + k_d) - \frac{1}{2}\sqrt{(v + p_{tot} + k_d)^2 - 4vp_{tot}}$, feasible under the physical constraints $0 \le c \le p_{tot}$. We have that Assumption 3 is  satisfied since $\frac{\partial f_z}{\partial z}$ is negative.
Thus, we obtain the reduced system

\small 
\begin{align*}
	\frac{d\bar{v}}{dt} &= k_1Z(t)(X_{tot} - \bar{v}) - k_2Y(\bar{v} - \bar{c}), \\
	\frac{d\bar{g}}{dt} &= \beta \bar{c} - \delta \bar{g},\\
\notag	\frac{d\bar{\psi_{v}}}{dt} &= (-k_1Z(t) - k_2Y)\bar{\psi_{v}}  + k_2Y \bar{\psi_c}   + \sqrt{k_1Z(t)( X_{tot} - \bar{v})} \Gamma_1 - \sqrt{k_2Y(\bar{v} - \bar{c})}\Gamma_2,\\
	\frac{d\bar{\psi_g}}{dt} &= \beta \bar{\psi_c} - \delta \bar{\psi_g}+ \sqrt{\beta \bar{c}} \Gamma_3 - \sqrt{\delta \bar{g}}\Gamma_4,
\end{align*}
where 
\begin{align*}
	\bar{c} &= \frac{1}{2}(\bar{v} + p_{tot} + k_d) - \frac{1}{2}\sqrt{(\bar{v} + p_{tot} + k_d)^2 - 4\bar{v}p_{tot}},   \\
	\bar{\psi}_c &=  \frac{ (p_{tot} - \bar{c}) \bar{\psi}_v}{  (\bar{v}  +  p_{tot} - 2\bar{c} + k_d)}.
\end{align*}\normalsize

Fig. \ref{sim} includes the simulation results for the error in second moments of the stochastic fluctuations of $v$ and $g$. We use zero initial conditions for all variables and thus the first moment of the stochastic fluctuations remains zero at all times. The simulations are carried out with the Euler-Maruyama method and the sample means are calculated using $3 \times 10^6$ realizations. 
\begin{figure}[h]
	\centering
		\includegraphics[scale=0.45]{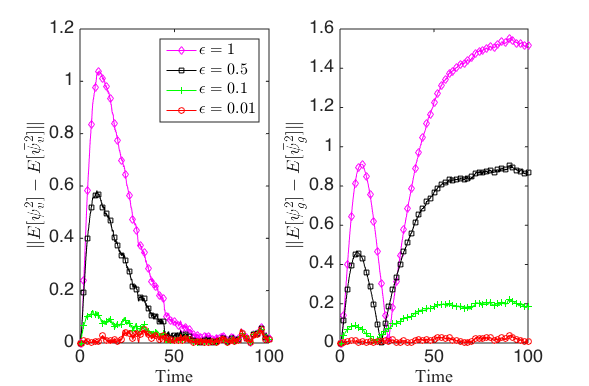}
	\vspace{-1em}
\caption{\small{Errors in the second moments decreases as $\epsilon$ decreases. The parameters used are $Z(t) = 1, \ k_1 = 0.01 , \ k_2 = 0.01, \ k_d = 100 , \ X_{tot} = 200, \ Y = 20, \ p_{tot} = 100, \ \delta = 0.1, \ \beta  = 0.1, v(0) = 0$, $c(0) = 0$, $g(0) = 0$, $\psi_v(0) = 0$, $\psi_g(0) = 0$.}}
\label{sim}
\end{figure}

\section{Conclusion}
\label{conclusion}
In this work, we obtained a reduced order model for the Linear Noise Approximation of biomolecular systems with separation in time-scales. It was shown that, for a finite time-interval the first and second moments of the reduced system are within an $O(\epsilon)$-neighborhood of the first and second moments of the slow variable dynamics of the original system. This result can be used to approximate the slow variable dynamics of the LNA with a system of reduced dimensions, which will be useful in analysis and simulations of biomolecular systems especially when the system has high dimension. The reduced model that we obtain is equivalent to the reduced order model derived in \cite{pahlajani2011stochastic}. Our results are also consistent with the error analysis that they have performed numerically, where it is approximated that the maximum errors in the mean and the variance over time are of $O(\epsilon)$.

In future work, we aim to extend this analysis to obtain an approximation for the fast variable dynamics.

\section*{Acknowledgements}
This work was funded by AFOSR grant  \# FA9550-14-1-0060.

\section*{Appendix}

\textbf{A-1: }Applying the coordinate transformation $x = A_xy$, $z= A_zy$ to equation (\ref{LNA_intro1}), with $\tilde{a}(y,t) = [\hat{a}_s(y,t), (1/\epsilon)\hat{a}_f(y,t)]^T$ and $v = [v_1, \ldots, v_{m_s}, v_{m_s + 1}, \ldots, v_{m_s + m_f}]$, with $y = A^{-1}[x,z]^T$ we have 

\small \vspace{-1em}
\begin{align}
\notag	\dot{x} &=  A_x f( A^{-1}[x,z]^T,t) \\& \notag = A_x \sum_{i = 1}^{m_s} v_i \hat{a}_{si}( A^{-1}[x,z]^T,t) +  A_x \sum_{i = {v_{m_s + 1}}}^{m_s + m_f} v_i (1/\epsilon)\hat{a}_{fi}( A^{-1}[x,z]^T,t) \\& \label{tr_eq1} = f_x(x,z,t),\\
\notag	\dot{z} &= A_z f( A^{-1}[x,z]^T,t) \\& \notag = A_z \sum_{i = 1}^{m_s} v_i \hat{a}_{si}( A^{-1}[x,z]^T,t)   +  A_z \sum_{i = {v_{m_s + 1}}}^{m_s + m_f} v_i (1/\epsilon)\hat{a}_{fi}( A^{-1}[x,z]^T,t)  \\& \label{tr_eq2} = \frac{1}{\epsilon}f_z(x,z,t,\epsilon).
\end{align}
\normalsize
Thus, from equation (\ref{tr_eq1}), if follows that $A_xv_i = 0$ for $i = m_s + 1, \ldots, m_s + m_f$.

Applying the coordinate transformation $\psi_x = A_x\xi$, $\psi_z= A_z\xi$, to equation (\ref{LNA_intro2}), we have that 

\small \vspace{0em}
\begin{align*}
	\dot{\psi_x} &=   A_x[A(y,t)\xi] + A_x\sigma(y,t) \Gamma,\\
	\dot{\psi_z} &=   A_z[A(y,t)\xi] + A_z\sigma(y,t) \Gamma.
\end{align*}
\normalsize
Since $A(y,t) = \frac{\partial f(y,t)}{\partial y}$ and $y = A^{-1}[x,z]^T$, using the chain rule we can write
\small{
\begin{align*}
	&\dot{\psi_x} 
	= A_x\left[\frac{\partial f( A^{-1}[x,z]^T,t)}{\partial x}\frac{\partial x}{ \partial  y} + \frac{\partial f( A^{-1}[x,z]^T,t)}{\partial z}\frac{\partial z}{ \partial  y}\right]\xi  \\& + A_x\left [v_1\sqrt{\tilde{a}_1( A^{-1}[x,z]^T,t)}, \ldots, v_m\sqrt{\tilde{a}_m( A^{-1}[x,z]^T,t)}\right] \Gamma,\\
	&\dot{\psi_z}  =  A_z\left[\frac{\partial f( A^{-1}[x,z]^T,t)}{\partial x}\frac{\partial x}{ \partial  y} + \frac{\partial f( A^{-1}[x,z]^T,t)}{\partial z}\frac{\partial z}{ \partial y}\right]\xi  \\& + A_z\left [v_1\sqrt{\tilde{a}_1( A^{-1}[x,z]^T,t)}, \ldots, v_m\sqrt{\tilde{a}_m( A^{-1}[x,z]^T,t)}\right] \Gamma.
\end{align*}} \normalsize
Using the linearity of the differentiation operator and the transformation $x = A_xy$, $z= A_zy$, we obtain 
\small{
\begin{align*}
	& \dot{\psi_x} = \left[\frac{\partial A_x f( A^{-1}[x,z]^T,t)}{\partial x}A_x + \frac{\partial A_x f( A^{-1}[x,z]^T,t)}{\partial z}A_z\right]\xi \\&  + A_x\left [v_1\sqrt{\tilde{a}_1( A^{-1}[x,z]^T,t)}, \ldots, v_m\sqrt{\tilde{a}_m( A^{-1}[x,z]^T,t)}\right] \Gamma,\\
	& \dot{\psi_z} = \left[\frac{\partial A_z f( A^{-1}[x,z]^T,t)}{\partial x}A_x  + \frac{\partial A_z f( A^{-1}[x,z]^T,t)}{\partial z}A_z\right]\xi  \\& + A_z\left [v_1\sqrt{\tilde{a}_1( A^{-1}[x,z]^T,t)}, \ldots, v_m\sqrt{\tilde{a}_m( A^{-1}[x,z]^T,t)}\right] \Gamma.
\end{align*}} \normalsize
From (\ref{tr_eq1}) - (\ref{tr_eq2}), we have that $ A_x f( A^{-1}[x,z]^T,t)  = f_x(x,z,t)$ and $ A_z f( A^{-1}[x,z]^T,t)  = \frac{1}{\epsilon}f_z(x,z,t,\epsilon)$. Furthermore, substituting for $\tilde{a}( A^{-1}[x,z]^T,t) =  [\hat{a}_s( A^{-1}[x,z]^T,t), (1/\epsilon)\hat{a}_f( A^{-1}[x,z]^T,t)]^T$, we have
\small{
\begin{align}
\notag	&\dot{\psi_x} 
	= \frac{\partial f_x(x,z,t)}{\partial x}\psi_x + \frac{\partial f_x(x,z,t)}{\partial z}\psi_z + \\&  \notag A_x\left [v_1\sqrt{\hat{a}{_s}_1(A^{-1}[x,z]^T,t)}, \ldots, v_{m_s}\sqrt{\hat{a}{_s}_{m_s}(A^{-1}[x,z]^T,t)}\right] \Gamma_x  \\& \label{tr_eq3} + A_x \bigg[v_{{m_s} + 1}\sqrt{\frac{1}{\epsilon}\hat{a}{_{f1}}(A^{-1}[x,z]^T,t)}, \ldots,   v_{m_s + m_f}\sqrt{\frac{1}{\epsilon} \hat{a}{_{fm_f}}(A^{-1}[x,z]^T,t)}\bigg] \Gamma_f,\\
\notag	&\dot{\psi_z} = \frac{\partial \frac{1}{\epsilon}f_z(x,z,t,\epsilon)}{\partial x}\psi_x + \frac{\partial \frac{1}{\epsilon} f_z(x,z,t,\epsilon)}{\partial z}\psi_z + \\& \notag A_z\left [v_1\sqrt{\hat{a}{_s}_1(A^{-1}[x,z]^T,t)}, \ldots, v_{m_s}\sqrt{\hat{a}{_s}_{m_s}(A^{-1}[x,z]^T,t)}\right] \Gamma_x \\& \label{tr_eq4} + A_z\bigg[v_{m_s +1} \sqrt{ \frac{1}{\epsilon} \hat{a}{_f}_1(A^{-1}[x,z]^T,t)}, \ldots,  v_{{m_s +m_f}}\sqrt{ \frac{1}{\epsilon} \hat{a}{_f}_{m_f}(A^{-1}[x,z]^T,t)}\bigg] \Gamma_f ,
\end{align}}
\normalsize
where $\Gamma = [\Gamma_x, \Gamma_f]^T$. From (\ref{tr_eq1}) we have that, $A_x v_i = 0$ for $i = m_s+1, \ldots, m_s+m_f$. Then, multiplying (\ref{tr_eq4}) by $\epsilon$, we can write the system (\ref{tr_eq3}) - (\ref{tr_eq4}), in the form of system (\ref{sys_ori3_cl}) - (\ref{sys_ori4_cl}), where $\Gamma_z = [\Gamma_x, \Gamma_f]^T$.

\bibliographystyle{unsrt}
\bibliography{CDC2016_v2}
\end{document}